\documentclass[a4paper]{amsart}
\usepackage[T1]{fontenc}
\usepackage[british]{babel}
\usepackage{amssymb}
\usepackage{enumitem}
\setlist[enumerate]{label=\arabic*)}
\usepackage{braket}
\usepackage{bm}
\usepackage{stmaryrd}
\let\originalleft\left
\let\originalright\right
\renewcommand{\left}{\mathopen{}\mathclose\bgroup\originalleft}
\renewcommand{\right}{\aftergroup\egroup\originalright}
\DeclareSymbolFont{bbold}{U}{bbold}{m}{n}
\DeclareSymbolFontAlphabet{\mathbbm}{bbold}
\newcommand{\M}{\mathfrak{M}}
\newcommand{\N}{\mathfrak{N}}
\newcommand{\B}{\mathbb{B}}
\newcommand{\0}{\mathbbm{0}}
\newcommand{\1}{\mathbbm{1}}
\newcommand{\Qp}[1]{{\left\llbracket #1 \right\rrbracket}}
\newcommand{\Seq}[1]{\Braket{\, #1 \,}}
\newcommand{\abs}[1]{{\left\lvert #1 \right\rvert}}
\newcommand{\pow}[1]{{\mathcal{P}\left( #1 \right)}}
\newcommand{\fin}[1]{{\left[ #1 \right]^{<\aleph_0}}}
\newcommand{\bp}[2]{{{#1}^{\left[#2\right]}}}
\newcommand{\bu}[3]{{{#1}^{\left[#2\right]}\! / #3}}
\newcommand{\eq}[2]{{\left[ #1 \right]_{#2}}}
\DeclareMathOperator{\dom}{dom}
\theoremstyle{plain}
\newtheorem{theorem}{Theorem}[section]
\newtheorem{lemma}[theorem]{Lemma}
\newtheorem{claim}{Claim}
\theoremstyle{definition}
\newtheorem{definition}[theorem]{Definition}
\newtheorem{question}[theorem]{Question}
\theoremstyle{remark}
\newtheorem{remark}[theorem]{Remark}
\begin{document}
\title{Keisler's order via Boolean ultrapowers}
\author{Francesco Parente}
\address{School of Mathematics\\University of Leeds\\Leeds LS2 9JT\\United Kingdom}
\thanks{This work is part of my PhD thesis at the University of East Anglia. I would like to thank my supervisors Mirna D\v zamonja and David Asper\'o for their invaluable guidance and support throughout my doctoral studies. I wish to express my gratitude also to Maryanthe Malliaris for her encouragement and advice, in particular during a research visit to the University of Chicago in May 2018.\\This research was partially supported by an Early Career Fellowship from the London Mathematical Society.}
\begin{abstract}
In this paper, we provide a new characterization of Keisler's order in terms of saturation of Boolean ultrapowers. To do so, we apply and expand the framework of `separation of variables' recently developed by Malliaris and Shelah. We also show that good ultrafilters on Boolean algebras are precisely the ones which capture the maximum class in Keisler's order, answering a question posed by Benda in 1974.
\end{abstract}
\maketitle

\section{Introduction}\label{section:uno}

This paper continues the research initiated in \cite{parente:oru}. We are interested in Keisler's order, a tool to classify complete theories according to saturation of ultrapowers. A central notion for the study of ultrapowers is the one of \emph{regular} ultrafilter.

\begin{definition}[Keisler \cite{keisler:regular}]\label{definition:regularpow} Let $\kappa$ be a cardinal. A filter $F$ over a set $I$ is \emph{$\kappa$-regular} if there exists a family $\Set{X_\alpha | \alpha<\kappa}\subseteq F$ such that for every infinite $I\subseteq\kappa$ we have $\bigcap_{\alpha\in I}X_\alpha=\emptyset$.
\end{definition}

Although regular ultrafilters were originally considered by Frayne, Morel and Scott \cite{fms:reduced} to determine the possible cardinalities of ultrapowers, they later played a prominent role in the classification of theories, starting from the following result by Keisler.

\begin{theorem}[{Keisler \cite[Corollary 2.1a]{keisler:notsat}}]\label{theorem:kindep} Let $\kappa$ be an infinite cardinal; suppose $U$ is a $\kappa$-regular ultrafilter over a set $I$. If two $L$-structures $\M$ and $\N$ are elementarily equivalent and $\abs{L}\le\kappa$, then
\[
\M^I\!/U\text{ is }\kappa^+\text{-saturated}\iff\N^I\!/U\text{ is }\kappa^+\text{-saturated}.
\]
\end{theorem}

In other words, Theorem \ref{theorem:kindep} implies that the saturation of the regular ultrapower of a model of a complete theory does not depend on the choice of the particular model, but only on the theory itself. This naturally suggests a pre-ordering on the class of complete theories, which is now known as Keisler's order.

\begin{definition}[Keisler \cite{keisler:notsat}]\label{definition:ko} Let $T_0$ and $T_1$ be complete countable theories and $\kappa$ a cardinal. We define \emph{$T_0\trianglelefteq_\kappa T_1$} if for every $\kappa$-regular ultrafilter $U$ over $\kappa$ and models $\M_0\models T_0$, $\M_1\models T_1$, if ${\M_1}^\kappa\!/U$ is $\kappa^+$-saturated then ${\M_0}^\kappa\!/U$ is $\kappa^+$-saturated.

\emph{Keisler's order} is then defined as follows: $T_0\trianglelefteq T_1$ if and only if $T_0\trianglelefteq_\kappa T_1$ holds for every cardinal $\kappa$.
\end{definition}

The main contributions of this paper are the following. In Section \ref{section:due}, we answer a question posed by Benda \cite{benda:ultrapowers} by showing that $\aleph_1$-incomplete $\lambda$-good ultrafilters, as defined in Mansfield \cite{MANSFIELD}, are precisely the ones which $\lambda$-saturate every Boolean ultrapower.

Section \ref{section:tre} provides a new characterization of Keisler's order in terms of saturation of Boolean ultrapowers. Namely, using a suitable notion of regularity for ultrafilters on complete Boolean algebras, we prove that the use of Boolean ultrapowers instead of ultrapowers gives exactly the same classification of theories as Keisler's order. Theorem \ref{theorem:kobu} was first presented at the Logic Colloquium 2017 in Stockholm, and announced in a previous paper \cite{parente:oru}. We present the proof here for the first time. Ulrich \cite{ulrich:bv} has obtained a similar formulation of Keisler's order using Boolean-valued models, but our work is completely independent.

\section{Good ultrafilters}\label{section:due}

This section is focused on the interaction between combinatorial properties of ultrafilters and model-theoretic properties of Boolean ultrapowers, in particular saturation. We start by introducing some useful terminology in the context of Boolean algebras.

\begin{definition} Let $\B$ be a Boolean algebra. An \emph{antichain} in $\B$ is a subset $A\subseteq\B\setminus\{\0\}$ such that for all $a,b\in A$, if $a\wedge b>\0$ then $a=b$.

Let $\lambda$ be a cardinal; a Boolean algebra $\B$ is \emph{$\lambda$-c.c.} if every antichain in $\B$ has cardinality less than $\lambda$.
\end{definition}

Let $A$ and $W$ be maximal antichains in a complete Boolean algebra $\B$. We say that $W$ is a \emph{refinement} of $A$ if for every $w\in W$ there exists $a\in A$ such that $w\le a$. Note that this element $a\in A$ is unique.

\begin{definition}[Hamkins and Seabold \cite{HAMSEA}]\label{definition:hs} Let $M$ be a set, $A$ a maximal antichain, and $\tau\colon A\to M$ a function. If $W$ is a refinement of $A$, the \emph{reduction} of $\tau$ to $W$ is the function
\[
\begin{split}
(\tau\mathbin{\downarrow} W)\colon W &\longrightarrow M \\
w &\longmapsto \tau(a)
\end{split}\ ,
\]
where $a$ is the unique element of $A$ such that $w\le a$.
\end{definition}

\begin{remark} Finitely many maximal antichains $A_0,\dots,A_{n-1}$ always admit a common refinement, namely the maximal antichain
\[
\bigwedge_{i<n}A_i=\Set{a_0\wedge\dots\wedge a_{n-1} | a_i\in A_i\text{ for }i<n}\setminus\{\0\}.
\]
\end{remark}

Furthermore, if $x\in\B$ and $A$ is a maximal antichain in $\B$, we say that $x$ is \emph{based} on $A$ if for every $a\in A$ either $a\le x$ or $a\wedge x=\0$.

Next, we introduce two properties of functions which will be crucial both in Definition \ref{definition:good} and in Definition \ref{definition:moral}. As usual, $\fin{\kappa}$ denotes the set of finite subsets of $\kappa$; later on ${[\kappa]}^n$ will denote the set of subsets of $\kappa$ of cardinality $n$.

\begin{definition} Let $\kappa$ be a cardinal, $\B$ a Boolean algebra, and $f\colon\fin{\kappa}\to\B$.
\begin{itemize}
\item $f$ is \emph{monotonic} if for all $S,T\in\fin{\kappa}$, $S\subseteq T$ implies $f(T)\le f(S)$;
\item $f$ is \emph{multiplicative} if for all $S,T\in\fin{\kappa}$, $f(S\cup T)=f(S)\wedge f(T)$.
\end{itemize}
\end{definition}

Thus, in this context, a monotonic function is in fact monotonically decreasing. We also note that every multiplicative function is monotonic.

The next lemma highlights a property of multiplicative functions, which will be useful to establish Theorem \ref{theorem:moralsat}. The starting point is the observation that, given a multiplicative function $g\colon\fin{\kappa}\to\B$ and a maximal antichain $A$, if $g(\{\alpha\})$ is based on $A$ for every $\alpha<\kappa$, then also $g(S)$ is based on $A$ for every $S\in\fin{\kappa}$.

\begin{lemma}\label{lemma:bound} Let $\kappa$ be a cardinal and $\B$ a complete Boolean algebra. For a multiplicative function $g\colon\fin{\kappa}\to\B$, the following two conditions are equivalent:
\begin{enumerate}
\item\label{lemma:bounduno} $\displaystyle{\bigwedge\Set{\bigvee \Set{g(S) | S\in{[\kappa]}^n}| n<\omega}=\0}$.\\
\item\label{lemma:bounddue} There is a maximal antichain $A\subset\B$ such that:
\begin{itemize}
\item for every $\alpha<\kappa$, $g(\{\alpha\})$ is based on $A$;
\item for every $a\in A$, the set $\Set{\alpha<\kappa | a\le g(\{\alpha\})}$ is finite.
\end{itemize}
\end{enumerate}
\end{lemma}
\begin{proof} $\ref{lemma:bounduno}\Longrightarrow\ref{lemma:bounddue}$ The idea for the proof of this implication is already implicit in Mansfield \cite[Theorem 4.1]{MANSFIELD}. Let $g\colon\fin{\kappa}\to\B$ be a multiplicative function satisfying \ref{lemma:bounduno}. Without loss of generality, we may assume that $g(\emptyset)=\1$. Let $D$ be the set of all $d\in\B\setminus\{\0\}$ such that: for every $\alpha<\kappa$, either $d\le g(\{\alpha\})$ or $d\wedge g(\{\alpha\})=\0$, and the set $\Set{\alpha<\kappa | d\le g(\{\alpha\})}$ is finite.
We shall show that $D$ is dense, so that every maximal antichain $A\subseteq D$ will have the desired property.

Let $b\in\B\setminus\{\0\}$; we need to find some $d\in D$ such that $d\le b$. For every $n<\omega$, let us define
\[
c_n=\bigvee \Set{g(S) | S\in{[\kappa]}^n}
\]
and note that, by hypothesis, $\bigwedge_{n<\omega}c_n=\0$. Furthermore, it is easy to verify that $c_{n+1}\le c_n$ for all $n<\omega$. The sequence $\Seq{c_n|n<\omega}$ may be eventually $\0$, however it is not identically $\0$, as $c_0=g(\emptyset)=\1$. It follows that there exists some $i<\omega$ such that $\0<b\wedge c_i\wedge\neg c_{i+1}$. Therefore, by definition of $c_i$, there exists $S\in{[\kappa]}^i$ such that
\[
d=b\wedge g(S)\wedge\neg c_{i+1}>\0.
\]
Clearly $d\le b$, so we shall conclude the proof by showing that $d\in D$. For every $\alpha<\kappa$, if $\alpha\in S$ then
\[
d\le g(S)\le g(\{\alpha\});
\]
otherwise, if $\alpha\notin S$, then by the multiplicativity of $g$
\[
d\wedge g(\{\alpha\})=b\wedge g(S)\wedge g(\{\alpha\})\wedge\neg c_{i+1}=b\wedge g(S\cup\{\alpha\})\wedge\neg c_{i+1}\le b\wedge c_{i+1}\wedge\neg c_{i+1}=\0.
\]
Therefore $d\in D$, as desired.

$\ref{lemma:bounddue}\Longrightarrow\ref{lemma:bounduno}$ Suppose $g\colon\fin{\kappa}\to\B$ satisfies \ref{lemma:bounddue}; let $A$ be the maximal antichain in $\B$ given by the hypothesis. We want to show that
\[
\bigwedge\Set{\bigvee \Set{g(S) | S\in{[\kappa]}^n}| n<\omega}=\0.
\]
Suppose not; then there exists some $a\in A$ such that
\[
a\wedge\bigwedge\Set{\bigvee \Set{g(S) | S\in{[\kappa]}^n}| n<\omega}>\0,
\]
hence for every $n<\omega$ there exists some $S\in{[\kappa]}^n$ such that $\0<a\wedge g(S)$. Using the fact that $g$ is monotonic, we note that $g(S)\le\bigwedge_{\alpha\in S}g(\{\alpha\})$ and consequently
\[
\0<a\wedge\bigwedge_{\alpha\in S}g(\{\alpha\}).
\]
Since each $g(\{\alpha\})$ is based on $A$, we conclude that for all $n<\omega$ there exists $S\in{[\kappa]}^n$ such that
\[
a\le\bigwedge_{\alpha\in S}g(\{\alpha\}),
\]
but this contradicts our condition \ref{lemma:bounddue}.
\end{proof}

We are ready to introduce the notion of \emph{goodness} for filters over sets, which is due to Keisler.

\begin{definition}[Keisler \cite{keisler:good}]\label{definition:good} Let $\lambda$ be a cardinal. A filter $F$ over a set $I$ is \emph{$\lambda$-good} if for every $\kappa<\lambda$ and every monotonic function $f\colon\fin{\kappa}\to F$, there exists a multiplicative function $g\colon\fin{\kappa}\to F$ with the property that $g(S)\subseteq f(S)$ for all $S\in\fin{\kappa}$.
\end{definition}

In our context, we shall be concerned with $\lambda$-good ultrafilters which have the additional property of being \emph{$\aleph_1$-incomplete}, i.e.\ not $\aleph_1$-complete. This is due to the following result, which shows that such ultrafilters are precisely the ones which yield $\lambda$-saturated ultrapowers.

\begin{theorem}[Keisler \cite{keisler:ultsat}]\label{theorem:keislerultsat} Let $\lambda$ be an uncountable cardinal. For an ultrafilter $U$ over a set $I$, the following conditions are equivalent:
\begin{itemize}
\item $U$ is $\aleph_1$-incomplete and $\lambda$-good; 
\item for every $L$-structure $\M$ with $\abs{L}<\lambda$, the ultrapower $\M^I\!/U$ is $\lambda$-saturated.
\end{itemize}
\end{theorem}

In the remainder of this section, we shall consider whether a similar characterization can be found for ultrafilters on complete Boolean algebras. Hence, we naturally look at Boolean ultrapowers, a generalization of the usual ultrapower construction. A detailed presentation of Boolean ultrapowers can be found elsewhere, such as in the standard reference of Mansfield \cite{MANSFIELD}. However, to keep this paper self-contained, we now recall the main points.

Let $\M$ be an $L$-structure and $\B$ a complete Boolean algebra. Define first the set of \emph{names}
\[
\bp{M}{\B}=\Set{\tau\colon A\to M | A\subset\B\text{ is a maximal antichain}}.
\]
\begin{remark} In the definition of names, we could equivalently reverse the arrows and consider functions from $M$ to $\B$, as in Mansfield \cite{MANSFIELD}. However, we find the above presentation more convenient for our purposes.
\end{remark}

We then introduce a Boolean-valued semantic as follows: let $\varphi(x_1,\dots,x_n)$ be an $L$-formula and $\tau_1,\dots,\tau_n\in \bp{M}{\B}$. If $W$ is any common refinement of $\dom(\tau_1),\dots,\dom(\tau_n)$, then
\[
\Qp{\varphi(\tau_1,\dots,\tau_n)}^\bp{\M}{\B}=\bigvee\Set{w\in W | \M\models\varphi\bigl((\tau_1\mathbin{\downarrow} W)(w),\dots,(\tau_n\mathbin{\downarrow} W)(w)\bigr)}.
\]
From now on, when there is no danger of confusion, the superscript $\bp{\M}{\B}$ will be omitted.

Given an ultrafilter $U$ on $B$, let $\equiv_U$ be the equivalence relation on $\bp{M}{\B}$ defined by
\[
\tau\equiv_U\sigma\overset{\mathrm{def}}{\iff} \Qp{\tau=\sigma}\in U.
\]
Quotienting the set of names by the above equivalence relation gives rise to the $L$-structure $\bu{\M}{\B}{U}$, the \emph{Boolean ultrapower} of $\M$ by $U$, which satisfies the following analogue of \L o\'s theorem.

\begin{theorem}[{Mansfield \cite[Theorem 1.5]{MANSFIELD}}]\label{theorem:los} Let $\M$ be an $L$-structure, $\B$ a complete Boolean algebra, and $U$ an ultrafilter on $\B$. For every $L$-formula $\varphi(x_1,\dots,x_n)$ and names $\tau_1,\dots,\tau_n\in\bp{M}{\B}$ we have
\[
\bu{\M}{\B}{U}\models\varphi\bigl(\eq{\tau_1}{U},\dots,\eq{\tau_n}{U}\bigr)\iff\Qp{\varphi(\tau_1,\dots,\tau_n)}\in U.
\]
\end{theorem}

The problem of finding a translation of Theorem \ref{theorem:keislerultsat} for Boolean ultrapowers was first considered by Mansfield \cite{MANSFIELD}, who defined good ultrafilters in a way formally analogous to Definition \ref{definition:good}.

\begin{definition}[Mansfield \cite{MANSFIELD}]\label{definition:mansfieldgood} Let $\lambda$ be a cardinal. An ultrafilter $U$ on a complete Boolean algebra $\B$ is \emph{$\lambda$-good} if for every $\kappa<\lambda$ and every monotonic function $f\colon\fin{\kappa}\to U$, there exists a multiplicative function $g\colon\fin{\kappa}\to U$ with the property that $g(S)\le f(S)$ for all $S\in\fin{\kappa}$.
\end{definition}

Using this definition, he was able to generalize one of the two implications of Theorem \ref{theorem:keislerultsat}.

\begin{theorem}[{Mansfield \cite[Theorem 4.1]{MANSFIELD}}]\label{theorem:goodmansfield} Let $\lambda$ be an uncountable cardinal, $\B$ a complete Boolean algebra, and $U$ an ultrafilter on $\B$. If $U$ is $\aleph_1$-incomplete and $\lambda$-good, then for every $L$-structure $\M$ with $\abs{L}<\lambda$, the Boolean ultrapower $\bu{\M}{\B}{U}$ is $\lambda$-saturated.
\end{theorem}

Three years later, Benda \cite{benda:ultrapowers} observed that, using Mansfield's definition, `it is not straightforward to prove the other implication'. To get around this problem, Benda introduced another class of ultrafilters, which he also called `$\lambda$-good'. To avoid a clash of terminology, we shall temporarily rename such ultrafilters `$\lambda$-Benda'. 

\begin{definition}[Benda \cite{benda:ultrapowers}] Let $\lambda$ be a cardinal. An ultrafilter $U$ on a complete Boolean algebra $\B$ is \emph{$\lambda$-Benda} if for every $\kappa<\lambda$, for every function $f\colon\fin{\kappa}\to U$, and every family of maximal antichains $\Set{A_\alpha | \alpha<\kappa}$, if for all $S\in\fin{\kappa}$ $f(S)$ is based on $\bigwedge_{\alpha\in S}A_\alpha$, then there exist a multiplicative function $g\colon\fin{\kappa}\to U$ and a maximal antichain $A$ satisfying:
\begin{enumerate}
\item for all $S\in\fin{\kappa}$, $g(S)$ is based on $A\wedge\bigwedge_{\alpha\in S}A_\alpha$;
\item for all $a\in A$, the set $\Set{\alpha<\kappa | a\wedge g(\{\alpha\})>\0}$ is finite;
\item for all $S\in\fin{\kappa}$, $g(S)\le f(S)$.
\end{enumerate}
\end{definition}

The above definition, although quite complex, was specifically designed to establish the following equivalence.

\begin{theorem}[Benda \cite{benda:ultrapowers}]\label{theorem:goodbenda} Let $\lambda$ be an uncountable cardinal. For an ultrafilter $U$ on a complete Boolean algebra $\B$, the following conditions are equivalent:
\begin{itemize}
\item $U$ is $\lambda$-Benda; 
\item for every $L$-structure $\M$ with $\abs{L}<\lambda$, the Boolean ultrapower $\bu{\M}{\B}{U}$ is $\lambda$-saturated.
\end{itemize}
\end{theorem}

Clearly, the combination of Theorem \ref{theorem:goodmansfield} and Theorem \ref{theorem:goodbenda} shows that if an ultrafilter $U$ is $\aleph_1$-incomplete and $\lambda$-good, then it is $\lambda$-Benda. However, the question whether Benda's notion is actually weaker remained open.

In 1982, Balcar and Franek \cite{balfran:cba} discussed the existence of independent families in complete Boolean algebras. They acknowledged the existence of two different definitions of goodness for ultrafilters on complete Boolean algebras, and noted that Mansfield's definition `is apparently stronger than Benda's for it implies that even the Boolean valued model of set theory modulo a $\kappa$-good ultrafilter is $\kappa$-saturated'. Indeed, it is true that if an ultrafilter $U$ on $\B$ satisfies Mansfield's definition of $\kappa$-goodness, then the quotient of any full $\B$-valued model by $U$ is $\kappa$-saturated. The reader can find a proof of this fact in \cite[Theorem 2.2.5]{parente:bvm}.

Motivated by Benda's question and Balcar and Franek's remark, we decided to further investigate this problem. In the main result of this section, we finally settle this question by showing that Mansfield's definition, although apparently stronger, is in fact equivalent to Benda's definition.

\begin{theorem}\label{theorem:mansfieldbenda} Let $\lambda$ be an uncountable cardinal, $\B$ be a complete Boolean algebra, and $U$ an ultrafilter on $\B$. Then $U$ is $\lambda$-Benda if and only if $U$ is $\aleph_1$-incomplete and $\lambda$-good.
\end{theorem}
\begin{proof} As we have already observed, one implication follows immediately combining Theorem \ref{theorem:goodmansfield} and Theorem \ref{theorem:goodbenda}.

For the other implication, suppose $U$ is $\lambda$-Benda: we shall show that $U$ is $\aleph_1$-incomplete and $\lambda$-good. Let $\kappa<\lambda$ be a cardinal which, without loss of generality, we may assume to be infinite.

Firstly, consider the constant function $c\colon\fin{\kappa}\to U$ defined by $c(S)=\1$ for all $S\in\fin{\kappa}$. By taking $A_\alpha=\{\1\}$ for each $\alpha<\kappa$, it is trivial that for all $S\in\fin{\kappa}$ $c(S)$ is based on $\bigwedge_{\alpha\in S}A_\alpha=\{\1\}$. Therefore, we can use the hypothesis that $U$ is $\lambda$-Benda to obtain a multiplicative function $x\colon\fin{\kappa}\to U$ and a maximal antichain $A\subset\B$ such that:
\begin{enumerate}
\item\label{mansfieldbendauno} for all $S\in\fin{\kappa}$, $x(S)$ is based on $A$;
\item\label{mansfieldbendadue} for all $a\in A$, the set defined as $S(a)=\Set{\alpha<\kappa | a\le x(\{\alpha\})}$ is finite.
\end{enumerate}

To show that $U$ is $\aleph_1$-incomplete, we recall that $\kappa$ is infinite and observe that
\[
\bigwedge_{n<\omega}x(\{n\})=\0.
\]
Indeed, if we had $\bigwedge_{n<\omega}x(\{n\})>\0$ then we could find some $a\in A$ such that $a\wedge\bigwedge_{n<\omega}x(\{n\})>\0$. From property \ref{mansfieldbendauno} we would have $a\le\bigwedge_{n<\omega}x(\{n\})$, contradicting property \ref{mansfieldbendadue}. Therefore, $U$ is $\aleph_1$-incomplete.

To show that $U$ is $\lambda$-good, let $f\colon\fin{\kappa}\to U$ be a monotonic function. We claim that there exists a maximal antichain $W$ such that for each $S\in\fin{\kappa}$, $f(S)\wedge x(S)$ is based on $W$. Our claim will be proved once we show that the set
\[
D=\Set{d\in\B\setminus\{\0\} | \text{for all }S\in\fin{\kappa},\text{ either }d\le f(S)\wedge x(S)\text{ or }d\wedge f(S)\wedge x(S)=\0}
\]
is dense in $\B$: any maximal antichain $W\subseteq D$ will then have the desired property. Let $b\in\B\setminus\{\0\}$; we shall find some $d\in D$ such that $d\le b$. First, we can find some $a\in A$ such that $\0<a\wedge b$. Now by property \ref{mansfieldbendadue} the set $S(a)$ is finite, so let $P$ be a common refinement of the finitely many maximal antichains $\{f(R),\neg f(R)\}$ for $R\subseteq S(a)$. Let $p\in P$ be such that $\0<p\wedge a\wedge b$; then it is clear that $d=p\wedge a\wedge b$ is such that $d\le b$. To see that $d\in D$, suppose $S\in\fin{\kappa}$ has the property that $\0<d\wedge f(S)\wedge x(S)$. In particular, we have $\0<a\wedge x(S)$, but $x(S)$ is based on $A$ and therefore $a\le x(S)$. This implies, since $x$ is monotonic, that $S\subseteq S(a)$. We deduce that $f(S)$ is based on $P$ and therefore $p\le f(S)$. Putting everything together, we conclude that
\[
d\le p\wedge a\le f(S)\wedge x(S).
\]
This shows that $d\in D$ and completes the proof of the claim.

Now, letting $W_\alpha=W$ for each $\alpha<\kappa$, we deduce from our claim that for all $S\in\fin{\kappa}$, $f(S)\wedge x(S)$ is based on $\bigwedge_{\alpha\in S}W_\alpha=W$. Hence, we can use the hypothesis that $U$ is $\lambda$-Benda to obtain a multiplicative function $g\colon\fin{\kappa}\to U$ such that for all $S\in\fin{\kappa}$, $g(S)\le f(S)\wedge x(S)\le f(S)$, as desired. Therefore, $U$ is $\lambda$-good and the proof is complete.
\end{proof}

In conclusion, Definition \ref{definition:mansfieldgood} can be regarded as an appropriate generalization of the notion of goodness to the context of complete Boolean algebras, for it implies that the parallel of Theorem \ref{theorem:keislerultsat} holds in full generality for Boolean ultrapowers.

\section{Keisler's order and saturation of Boolean ultrapowers}\label{section:tre}

In this section we present some applications of Boolean ultrapowers to the study of Keisler's order. Even though Keisler's order is defined via regular ultrafilters over sets, in the last years there has been a shift towards the construction of ultrafilters on complete Boolean algebras, using the framework of `separation of variables' developed by Malliaris and Shelah \cite{ms:dl}. In particular, moral ultrafilters (Definition~\ref{definition:moral}) have emerged as the main tool to find dividing lines among unstable theories. Motivated by these ideas, in Theorem \ref{theorem:kobu} we establish a characterization of Keisler's order via regular ultrafilters on complete Boolean algebras.

A preliminary remark on notation: when we introduce a formula as $\varphi(\bm{x})$, we mean that $\bm{x}$ is a finite tuple of variables including the ones appearing free in $\varphi$. If we then write $\varphi(\bm{a})$, we shall implicitly assume that $\bm{a}$ is a finite tuple of parameters of the same length as the tuple $\bm{x}$. By abuse of notation, tuples of functions will be sometimes treated as single functions, with the convention that if $\bm{\tau}=\langle\tau_1,\dots,\tau_n\rangle$, then $\bm{\tau}(b)=\langle\tau_1(b),\dots,\tau_n(b)\rangle$. Throughout this section, cardinals $\kappa$ and $\lambda$ will be assumed to be infinite.

We have previously analysed and compared two different ways in which Definition~\ref{definition:regularpow} can be generalized to ultrafilters on complete Boolean algebras. Here we just recall the relevant definition and refer the reader to \cite{parente:oru} for further details on regularity of ultrafilters.

\begin{definition}[Shelah \cite{sh:1064}]\label{definition:regular} Let $\kappa$ be a cardinal. An ultrafilter $U$ on a complete Boolean algebra $\B$ is \emph{$\kappa$-regular} if there exist a family $\Set{x_\alpha | \alpha<\kappa}\subseteq U$ and a maximal antichain $A\subset\B$ such that:
\begin{itemize}
\item for every $\alpha<\kappa$, $x_\alpha$ is based on $A$;
\item for every $a\in A$, the set $\Set{\alpha<\kappa | a\le x_\alpha}$ is finite.
\end{itemize}
\end{definition}

We now introduce the crucial concept of \emph{morality}, which can be thought of as a `local' version of goodness relative to some theory $T$. The meaning of `relative to $T$' is made precise in the definition of \emph{possibility}.

\begin{definition}\label{definition:possibility} Let $\kappa$ be a cardinal, $\B$ a complete Boolean algebra, $T$ a complete countable theory, and $\varphi=\Seq{\varphi_\alpha(x,\bm{y}_\alpha) | \alpha<\kappa}$ a sequence of formulae in the language of $T$.

A \emph{$\langle\kappa,\B,T,\varphi\rangle$-possibility} is a monotonic function $f\colon\fin{\kappa}\to\B\setminus\{\0\}$ such that: for all $S_*\in\fin{\kappa}$ and $a\in\B\setminus\{\0\}$ which satisfy:
\begin{itemize}
\item for every $S\subseteq S_*$ either $a\le f(S)$ or $a\wedge f(S)=\0$,
\item $S_*\subseteq\Set{\alpha<\kappa | a\le f(\{\alpha\})}$,
\end{itemize}
there exist a model $\M\models T$ and $\Set{\bm{b}_\alpha|\alpha\in S_*}$ in $M$ such that for all $S\subseteq S_*$
\begin{equation}\label{eq:possibility}
a\le f(S)\iff\M\models\exists x\bigwedge_{\alpha\in S}\varphi_\alpha(x,\bm{b}_\alpha).
\end{equation}
\end{definition}

Following our convention, in Definition \ref{definition:possibility} it is implicitly assumed that each $\bm{b}_\alpha$ is a finite tuple from $M$ of the same length as $\bm{y}_\alpha$.

\begin{definition}[Malliaris and Shelah \cite{ms:dl}]\label{definition:moral} Let $\kappa$ be a cardinal, $\B$ a complete Boolean algebra, and $T$ a complete countable theory. An ultrafilter $U$ on $\B$ is \emph{$\langle\kappa,\B,T\rangle$-moral} if for every sequence of formulae $\varphi=\Seq{\varphi_\alpha(x,\bm{y}_\alpha) | \alpha<\kappa}$ and every $\langle\kappa,\B,T,\varphi\rangle$-possibility $f\colon\fin{\kappa}\to U$, there exists a multiplicative function $g\colon\fin{\kappa}\to U$ with the property that $g(S)\le f(S)$ for all $S\in\fin{\kappa}$.
\end{definition}

Moral ultrafilters have recently played a crucial role in the study of Keisler's order, due to Malliaris and Shelah's technique of separation of variables, which is summarized in the following theorem.

\begin{theorem}[{Malliaris and Shelah \cite[Theorem 6.13]{ms:dl}}]\label{theorem:separation} Let $\kappa$ be a cardinal and $\B$ a complete Boolean algebra. Suppose $j\colon\pow{\kappa}\to\B$ is a surjective homomorphism with the property that $j^{-1}[\{\1\}]$ is a $\kappa$-regular $\kappa^+$-good filter over $\kappa$. Let $U$ be any ultrafilter on $\B$. Then, for a complete countable theory $T$ the following conditions are equivalent:
\begin{itemize}
\item $U$ is $\langle\kappa,\B,T\rangle$-moral;
\item for a model $\M\models T$, the ultrapower $\M^\kappa\!/j^{-1}[U]$ is $\kappa^+$-saturated.
\end{itemize}
\end{theorem}

Note that $j^{-1}[U]$ is a $\kappa$-regular ultrafilter over $\kappa$, because it includes the $\kappa$-regular filter $j^{-1}[\{\1\}]$.

The following result, usually referred to as the `existence theorem', shows that, under some conditions on the Boolean algebra $\B$, a surjective homomorphism as in Theorem \ref{theorem:separation} does indeed exist.

\begin{theorem}[{Malliaris and Shelah \cite[Theorem 8.1]{ms:dl}}]\label{theorem:existence} Let $\kappa$ be a cardinal. If $\B$ is a $\kappa^+$-c.c.\ complete Boolean algebra of cardinality $\le 2^\kappa$, then there exists a surjective homomorphism $j\colon\pow{\kappa}\to\B$ such that $j^{-1}[\{\1\}]$ is a $\kappa$-regular $\kappa^+$-good filter over $\kappa$.
\end{theorem}

In order to formulate the main results of this section, it is convenient to introduce a natural concept of saturation for ultrafilters on complete Boolean algebras.

\begin{definition}\label{definition:saturates} Let $\lambda$ be a cardinal and $\B$ a complete Boolean algebra. Suppose $U$ is an ultrafilter on $\B$; we say that $U$ \emph{$\lambda$-saturates} a complete theory $T$ if for every $\lambda$-saturated model $\M\models T$, the Boolean ultrapower $\bu{\M}{\B}{U}$ is $\lambda$-saturated.
\end{definition}

\begin{remark} Suppose $U$ is a $\kappa$-regular ultrafilter over a set $I$. Then, by Theorem~\ref{theorem:kindep}, $U$ $\kappa^+$-saturates a complete countable theory $T$ if and only if for some (or, equivalently, for every) model $\M\models T$, the ultrapower $\M^I\!/U$ is $\kappa^+$-saturated. In his independent work, Ulrich \cite{ulrich:bv} fully generalized Theorem \ref{theorem:kindep} to regular ultrafilters on Boolean algebras; hence, the equivalence is still true when considering saturation of Boolean ultrapowers. However, to keep our proofs self-contained, we prefer to use Definition \ref{definition:saturates} and refer the reader to \cite[Theorem 5.9]{ulrich:bv} for a proof of the equivalence.
\end{remark}

Shelah \cite[Claim 3.4]{sh:1064} has first established a connection between morality of ultrafilters and saturation of Boolean ultrapowers. However, his result is framed in the context of atomic saturation in the infinitary logic $\mathbb{L}_{\theta\theta}$, and only the case where $\B$ is a power-set algebra is proved explicitly. In the next theorem we present a detailed explanation of the equivalence; our proof relies on Lemma \ref{lemma:bound}.

\begin{theorem}\label{theorem:moralsat} Let $\kappa$ be a cardinal, $\B$ a complete Boolean algebra, and $U$ a $\kappa$-regular ultrafilter on $\B$. Then, for a complete countable theory $T$ the following conditions are equivalent:
\begin{itemize}
\item $U$ is $\langle\kappa,\B,T\rangle$-moral;
\item $U$ $\kappa^+$-saturates $T$.
\end{itemize}
\end{theorem}
\begin{proof} Let the family $\Set{x_\alpha | \alpha<\kappa}\subseteq U$ and the maximal antichain $A\subset\B$ witness the $\kappa$-regularity of $U$. For each $b\in\B$, we define
\[
S(b)=\Set{\alpha<\kappa | b\le x_\alpha};
\]
by $\kappa$-regularity, the set $S(b)$ is finite whenever $b>\0$.

Suppose $U$ is $\langle\kappa,\B,T\rangle$-moral. Let $\M$ be a model of $T$ and
\[
p(x)=\Set{\varphi_\alpha\bigl(x,\eq{\bm{\tau}_\alpha}{U}\bigr) | \alpha<\kappa}
\]
be a type in $\bu{\M}{\B}{U}$, where each $\bm{\tau}_\alpha$ is a finite tuple from $\bp{M}{\B}$. We shall show that $p(x)$ is realized in $\bu{\M}{\B}{U}$.

By Theorem \ref{theorem:los}, for every $S\in\fin{\kappa}$ we have
\[
\Qp{\exists x\bigwedge_{\alpha\in S}\varphi_\alpha(x,\bm{\tau}_\alpha)}\in U.
\]
This allows us to define a monotonic function $f\colon\fin{\kappa}\to U$ by letting for every $S\in\fin{\kappa}$
\[
f(S)=\Qp{\exists x\bigwedge_{\alpha\in S}\varphi_\alpha(x,\bm{\tau}_\alpha)}\wedge\bigwedge_{\alpha\in S}x_\alpha.
\]
Define $\varphi=\Seq{\varphi_\alpha(x,\bm{y}_\alpha) | \alpha<\kappa}$, where each $\bm{y}_\alpha$ is a new tuple of variables of the same length as $\bm{\tau}_\alpha$; we aim to show that $f$ is a $\langle\kappa,\B,T,\varphi\rangle$-possibility. Let $S_*\in\fin{\kappa}$ and $a\in\B\setminus\{\0\}$ be fixed, and assume that
\begin{itemize}
\item for every $S\subseteq S_*$ either $a\le f(S)$ or $a\wedge f(S)=\0$,
\item $S_*\subseteq\Set{\alpha<\kappa | a\le f(\{\alpha\})}$.
\end{itemize}
Choose a maximal antichain $D$ with the following property: for each $\alpha\in S_*$, $D$ is a refinement of the domain of each name in the tuple $\bm{\tau}_\alpha$. Then there exists $d\in D$ such that $d\wedge a>\0$. For every $\alpha\in S_*$, according to Definition \ref{definition:hs} let
\[
\bm{b}_\alpha=(\bm{\tau}_\alpha\mathbin{\downarrow} D)(d);
\]
we show that this choice satisfies \eqref{eq:possibility} from Definition \ref{definition:possibility}. Let $S\subseteq S_*$; if $a\le f(S)$, then $\0<d\wedge f(S)\le d\wedge \Qp{\exists x\bigwedge_{\alpha\in S}\varphi_\alpha(x,\bm{\tau}_\alpha)}$, therefore
\begin{equation}\label{eq:moralposs}
\M\models\exists x\bigwedge_{\alpha\in S}\varphi_\alpha\bigl(x,(\bm{\tau}_\alpha\mathbin{\downarrow} D)(d)\bigr).
\end{equation}
Conversely, if \eqref{eq:moralposs} holds then $d\le\Qp{\exists x\bigwedge_{\alpha\in S}\varphi_\alpha(x,\bm{\tau}_\alpha)}$ and therefore
\[
a\wedge\Qp{\exists x\bigwedge_{\alpha\in S}\varphi_\alpha(x,\bm{\tau}_\alpha)}>\0.
\]
From the second assumption above we have $a\le\bigwedge_{\alpha\in S_*}x_\alpha\le\bigwedge_{\alpha\in S}x_\alpha$, so we deduce $a\wedge f(S)>\0$ and finally, from the first assumption above, $a\le f(S)$. This concludes the proof that $f$ is a $\langle\kappa,\B,T,\varphi\rangle$-possibility.

Since we are assuming that $U$ is $\langle\kappa,\B,T\rangle$-moral, there exists a multiplicative function $g\colon\fin{\kappa}\to U$ with the property that $g(S)\le f(S)$ for all $S\in\fin{\kappa}$. We wish to show that $g$ satisfies condition \ref{lemma:bounduno} of Lemma \ref{lemma:bound}. Suppose not; then there exists some $a\in A$ such that
\[
a\wedge\bigwedge\Set{\bigvee \Set{g(S) | S\in{[\kappa]}^n}| n<\omega}>\0,
\]
hence for every $n<\omega$ there exists some $S\in{[\kappa]}^n$ such that
\[
\0<a\wedge g(S)\le a\wedge f(S)\le a\wedge\bigwedge_{\alpha\in S} x_\alpha,
\]
but this contradicts our regularity assumption that $S(a)$ is finite. This proves that $g$ satisfies condition \ref{lemma:bounduno} of Lemma \ref{lemma:bound}.

Consequently, there exists a maximal antichain $W\subset\B$ such that:
\begin{itemize}
\item for every $\alpha<\kappa$, $g(\{\alpha\})$ is based on $W$;
\item for every $w\in W$, the set defined as $R(w)=\Set{\alpha<\kappa | w\le g(\{\alpha\})}$ is finite.
\end{itemize}
For each $w\in W$, use fullness of the Boolean-valued model $\bp{\M}{\B}$ (see Mansfield \cite[Theorem 1.4]{MANSFIELD}) to choose a name $\tau_w\in \bp{M}{\B}$ such that
\[
\Qp{\exists x\bigwedge_{\alpha\in R(w)}\varphi_\alpha(x,\bm{\tau}_\alpha)}=\Qp{\bigwedge_{\alpha\in R(w)}\varphi_\alpha(\tau_w,\bm{\tau}_\alpha)}.
\]
Finally, by \cite[Theorem 1.3]{MANSFIELD}, let $\tau\in \bp{M}{\B}$ be such that for all $w\in W$, $w\le\Qp{\tau=\tau_w}$. We shall show that $\eq{\tau}{U}$ realizes the type $p(x)$ in $\bu{\M}{\B}{U}$.

For every $w\in W$, by multiplicativity of $g$ we have
\begin{multline*}
w\le\bigwedge_{\alpha\in R(w)}g(\{\alpha\})=g(R(w))\le f(R(w))\\ \le\Qp{\exists x\bigwedge_{\alpha\in R(w)}\varphi_\alpha(x,\bm{\tau}_\alpha)}=\Qp{\bigwedge_{\alpha\in R(w)}\varphi_\alpha(\tau_w,\bm{\tau}_\alpha)},
\end{multline*}
whence
\[
w\le\Qp{\bigwedge_{\alpha\in R(w)}\varphi_\alpha(\tau_w,\bm{\tau}_\alpha)}\wedge\Qp{\tau=\tau_w}\le\Qp{\bigwedge_{\alpha\in R(w)}\varphi_\alpha(\tau,\bm{\tau}_\alpha)}.
\]
If follows that for every $\alpha<\kappa$
\[
\Qp{\varphi_\alpha(\tau,\bm{\tau}_\alpha)}\ge\bigvee\Set{w\in W | \alpha\in R(w)}=\bigvee\Set{w\in W | w\le g(\{\alpha\})}=g(\{\alpha\})\in U,
\]
thus showing that $\Qp{\varphi_\alpha(\tau,\bm{\tau}_\alpha)}\in U$. This completes the proof that $\bu{\M}{\B}{U}$ is $\kappa^+$-saturated.

For the reverse implication, suppose $U$ $\kappa^+$-saturates the complete theory $T$. Let $\varphi=\Seq{\varphi_\alpha(x,\bm{y}_\alpha) | \alpha<\kappa}$ be a sequence of formulae; for a $\langle\kappa,\B,T,\varphi\rangle$-possibility $f\colon\fin{\kappa}\to U$ we shall find a multiplicative function $g\colon\fin{\kappa}\to U$ with the property that $g(S)\le f(S)$ for all $S\in\fin{\kappa}$.

\begin{claim}\label{claim:wa} There exists a refinement $W$ of $A$ with the property that for every $w\in W$ and every $S\subseteq S(w)$ either $w\le f(S)$ or $w\wedge f(S)=\0$.
\end{claim}
\begin{proof}[Proof of Claim \ref{claim:wa}] Let $D$ be the set of all $d\in\B\setminus\{\0\}$ which are below some element of $A$, and such that for every $S\subseteq S(d)$ either $d\le f(S)$ or $d\wedge f(S)=\0$. We shall show that $D$ is dense in $\B$, so that every maximal antichain $W\subseteq D$ will have the desired property. The same argument as the proof of Theorem \ref{theorem:mansfieldbenda} will work: indeed, suppose $b\in\B\setminus\{\0\}$. We can find some $a\in A$ such that $a\wedge b>\0$. Now let $P$ be a common refinement of the finitely many maximal antichains $\{f(S),\neg f(S)\}$ for $S\subseteq S(a)$. Let $p\in P$ be such that $p\wedge a\wedge b>\0$; then it is clear that $d=p\wedge a\wedge b$ is such that $d\le b$ and $d\in D$.
\end{proof}

Now, for each $a\in W$ let
\[
S_*(a)=\Set{\alpha\in S(a) | a\le f(\{\alpha\})},
\]
and note that:
\begin{itemize}
\item for every $S\subseteq S_*(a)$ either $a\le f(S)$ or $a\wedge f(S)=\0$;
\item $S_*(a)\subseteq\Set{\alpha<\kappa | a\le f(\{\alpha\})}$.
\end{itemize}
Since $f$ is a $\langle\kappa,\B,T,\varphi\rangle$-possibility, for each $a\in W$ there exist a model $\M_a\models T$ and parameters $\Set{\bm{b}_\alpha(a)|\alpha\in S_*(a)}$ in $M_a$ such that for all $S\subseteq S_*(a)$
\begin{equation}\label{eq:apossibility}
a\le f(S)\iff\M_a\models\exists x\bigwedge_{\alpha\in S}\varphi_\alpha(x,\bm{b}_\alpha(a)).
\end{equation}
Now let $\M$ be a $\kappa^+$-saturated model of $T$.

\begin{claim}\label{claim:wb} For every $a\in W$ there exists a sequence $\Seq{\bm{\tau}_\alpha(a) | \alpha<\kappa}$ in $M$ such that for every $S\subseteq S_*(a)$
\[
a\le f(S)\iff\M\models\exists x\bigwedge_{\alpha\in S}\varphi_\alpha(x,\bm{\tau}_\alpha(a)).
\]
\end{claim}
\begin{proof}[Proof of Claim \ref{claim:wb}] Let us fix $a\in W$. We define
\begin{align*}
\Gamma_a={}&\Set{\exists x\bigwedge_{\alpha\in S}\varphi_\alpha(x,\bm{y}_\alpha) | S\subseteq S_*(a)\text{ and }a\le f(S)}\\ {}\cup{} &\Set{\lnot\exists x\bigwedge_{\alpha\in S}\varphi_\alpha(x,\bm{y}_\alpha) | S\subseteq S_*(a)\text{ and }a\wedge f(S)=\0}.
\end{align*}
Let $\bm{y}$ be the finite tuple of variables made of all the $\bm{y}_\alpha$ appearing in $\Gamma_a$. Then \eqref{eq:apossibility} implies that $\M_a\models\exists\bm{y}\bigwedge\Gamma_a$, but $\M_a\equiv\M$ by completeness of $T$, therefore $\M\models\exists\bm{y}\bigwedge\Gamma_a$. This allows us to define $\bm{\tau}_\alpha(a)$ in $M$ for every $\alpha\in S_*(a)$. Otherwise, if $\alpha\notin S_*(a)$, we can define $\bm{\tau}_\alpha(a)$ arbitrarily.
\end{proof}

We have thus defined a sequence of tuples of names $\Seq{\bm{\tau}_\alpha | \alpha<\kappa}$ in $\bp{M}{\B}$. We aim to prove that
\[
p(x)=\Set{\varphi_\alpha\bigl(x,\eq{\bm{\tau}_\alpha}{U}\bigr) | \alpha<\kappa}
\]
is a type in $\bu{\M}{\B}{U}$. To do so, we shall show that for each $S\in\fin{\kappa}$
\begin{equation} \label{eq:ptype}
\Qp{\exists x\bigwedge_{\alpha\in S}\varphi_\alpha(x,\bm{\tau}_\alpha)}\wedge\bigwedge_{\alpha\in S}\bigl(f(\{\alpha\})\wedge x_\alpha\bigr)= f(S)\wedge\bigwedge_{\alpha\in S}x_\alpha\in U
\end{equation}
and then conclude using Theorem \ref{theorem:los}. First of all, note that both sides of \eqref{eq:ptype} are based on $W$, due to our choice of $W$ in Claim \ref{claim:wa}. Hence,
\begin{multline*}
\Qp{\exists x\bigwedge_{\alpha\in S}\varphi_\alpha(x,\bm{\tau}_\alpha)}\wedge\bigwedge_{\alpha\in S}\bigl(f(\{\alpha\})\wedge x_\alpha\bigr) \\
=\bigvee\Set{a\in W | \M\models\exists x\bigwedge_{\alpha\in S}\varphi_\alpha(x,\bm{\tau}_\alpha(a))} \wedge\bigwedge_{\alpha\in S}\bigl(f(\{\alpha\})\wedge x_\alpha\bigr) \\
=\bigvee\Set{ a\wedge\bigwedge_{\alpha\in S}\bigl(f(\{\alpha\})\wedge x_\alpha\bigr)| a\in W,\ \M\models\exists x\bigwedge_{\alpha\in S}\varphi_\alpha(x,\bm{\tau}_\alpha(a))} \\
=\bigvee\Set{a\in W | a\le\bigwedge_{\alpha\in S}\bigl(f(\{\alpha\})\wedge x_\alpha\bigr),\ \M\models\exists x\bigwedge_{\alpha\in S}\varphi_\alpha(x,\bm{\tau}_\alpha(a))} \\
=\bigvee\Set{a\in W | S\subseteq S_*(a), \ \M\models\exists x\bigwedge_{\alpha\in S}\varphi_\alpha(x,\bm{\tau}_\alpha(a))} \\
=\bigvee\Set{a\in W | a\le f(S)\wedge\bigwedge_{\alpha\in S}x_\alpha}=f(S)\wedge\bigwedge_{\alpha\in S}x_\alpha,
\end{multline*}
thus showing, in particular, that $p(x)$ is finitely satisfiable, hence a type in $\bu{\M}{\B}{U}$.

Since we are assuming that $U$ $\kappa^+$-saturates $T$, let $\tau\in\bp{M}{\B}$ be a name such that $\eq{\tau}{U}$ realizes $p(x)$ in $\bu{\M}{\B}{U}$. We define a function $g\colon\fin{\kappa}\to U$ as follows: for $S\in\fin{\kappa}$,
\[
g(S)=\bigwedge_{\alpha\in S}\bigl(\Qp{\varphi_\alpha(\tau,\bm{\tau}_\alpha)}\wedge f(\{\alpha\})\wedge x_\alpha\bigr).
\]
Then clearly $g$ is multiplicative and, for every $S\in\fin{\kappa}$, we may apply \eqref{eq:ptype} to obtain
\begin{multline*}
g(S)=\Qp{\bigwedge_{\alpha\in S}\varphi_\alpha(\tau,\bm{\tau}_\alpha)}\wedge\bigwedge_{\alpha\in S}\bigl(f(\{\alpha\})\wedge x_\alpha\bigr)\\ \le\Qp{\exists x\bigwedge_{\alpha\in S}\varphi_\alpha(x,\bm{\tau}_\alpha)}\wedge\bigwedge_{\alpha\in S}\bigl(f(\{\alpha\})\wedge x_\alpha\bigr)\le f(S).
\end{multline*}
This completes the proof that $U$ is $\langle\kappa,\B,T\rangle$-moral.
\end{proof}

We now move on to present our main result in this section, which follows from Theorem \ref{theorem:moralsat} and Malliaris and Shelah's technique of separation of variables.

\begin{theorem}\label{theorem:kobu} Let $\kappa$ be a cardinal and $T_0$, $T_1$ complete countable theories. Then the following are equivalent:
\begin{itemize}
\item $T_0\trianglelefteq_\kappa T_1$;
\item for every $\kappa^+$-c.c.\ complete Boolean algebra $\B$ of cardinality $\le 2^\kappa$, and every $\kappa$-regular ultrafilter $U$ on $\B$, if $U$ $\kappa^+$-saturates $T_1$ then $U$ $\kappa^+$-saturates $T_0$.
\end{itemize}
\end{theorem}
\begin{proof} Suppose that $T_0\trianglelefteq_\kappa T_1$. Let $\B$ be a $\kappa^+$-c.c.\ complete Boolean algebra with $\abs{\B}\le 2^\kappa$, and let $U$ be a $\kappa$-regular ultrafilter on $\B$ which $\kappa^+$-saturates $T_1$. By Theorem \ref{theorem:moralsat}, we know that $U$ is $\langle\kappa,\B,T_1\rangle$-moral.

By Theorem \ref{theorem:existence}, there exists a surjective homomorphism $j\colon\pow{\kappa}\to\B$ such that $j^{-1}[\{\1\}]$ is a $\kappa$-regular $\kappa^+$-good filter over $\kappa$. Therefore, $j^{-1}[U]$ is a $\kappa$-regular ultrafilter over $\kappa$, which $\kappa^+$-saturates $T_1$ by Theorem \ref{theorem:separation}. But $T_0\trianglelefteq_\kappa T_1$, therefore $j^{-1}[U]$ also $\kappa^+$-saturates $T_0$. By Theorem \ref{theorem:separation} again, we deduce that $U$ is $\langle\kappa,\B,T_0\rangle$-moral, and finally we conclude that $U$ is $\kappa^+$-saturates $T_0$ by Theorem \ref{theorem:moralsat}.

For the reverse implication, it is sufficient to observe that $\pow{\kappa}$ is a $\kappa^+$-c.c.\ complete Boolean algebra of cardinality $\le 2^\kappa$.
\end{proof}

Working independently, Ulrich \cite{ulrich:bv} has obtained another formulation of Keisler's order using Boolean-valued models. Compared to our Theorem \ref{theorem:kobu}, his characterization holds for all ultrafilters on $\kappa^+$-c.c.\ complete Boolean algebras $\B$, without the assumption $\abs{\B}\le 2^\kappa$. The following question, however, remains open.

\begin{question} Does the equivalence of Theorem \ref{theorem:kobu} still hold without the $\kappa^+$-c.c.\ assumption on $\B$?
\end{question}


\begin{thebibliography}{10}

\bibitem{balfran:cba}
B.~Balcar and F.~Franek, \emph{Independent families in complete Boolean
  algebras}, Transactions of the American Mathematical Society \textbf{274}
  (1982), no.~2, 607--618.

\bibitem{benda:ultrapowers}
Miroslav Benda, \emph{Note on Boolean ultrapowers}, Proceedings of the
  American Mathematical Society \textbf{46} (1974), no.~2, 289--293.

\bibitem{fms:reduced}
T.~Frayne, A.~C. Morel, and D.~S. Scott, \emph{Reduced direct products},
  Fundamenta Mathematicae \textbf{51} (1962), no.~3, 195--228.

\bibitem{HAMSEA}
Joel~David Hamkins and Daniel~Evan Seabold, \emph{Well-founded Boolean
  ultrapowers as large cardinal embeddings}, arXiv:\href{http://arxiv.org/abs/1206.6075}{\texttt{1206.6075} \texttt{[math.LO]}}.

\bibitem{keisler:good}
H.~Jerome Keisler, \emph{Good ideals in fields of sets}, Annals of Mathematics
  \textbf{79} (1964), no.~2, 338--359.

\bibitem{keisler:regular}
H.~Jerome Keisler, \emph{On cardinalities of ultraproducts}, Bulletin of the American
  Mathematical Society \textbf{70} (1964), no.~4, 644--647.

\bibitem{keisler:ultsat}
H.~Jerome Keisler, \emph{Ultraproducts and saturated models}, Indagationes Mathematicae
  (Proceedings) \textbf{67} (1964), 178--186.

\bibitem{keisler:notsat}
H.~Jerome Keisler, \emph{Ultraproducts which are not saturated}, The Journal of Symbolic
  Logic \textbf{32} (1967), no.~1, 23--46.

\bibitem{ms:dl}
M.~Malliaris and S.~Shelah, \emph{A dividing line within simple unstable
  theories}, Advances in Mathematics \textbf{249} (2013), 250--288.

\bibitem{MANSFIELD}
Richard Mansfield, \emph{The theory of Boolean ultrapowers}, Annals of
  Mathematical Logic \textbf{2} (1971), no.~3, 297--323.

\bibitem{parente:bvm}
Francesco Parente, \emph{Boolean valued models, saturation, forcing axioms},
  Master's thesis, Universit{\`a} di Pisa, 2015.

\bibitem{parente:oru}
Francesco Parente, \emph{On regular ultrafilters, Boolean ultrapowers, and Keisler's
  order}, RIMS K{\=o}ky{\=u}roku \textbf{2081} (2018), 41--56.

\bibitem{sh:1064}
Saharon Shelah, \emph{Atomic saturation of reduced powers}, arXiv:\href{http://arxiv.org/abs/1601.04824}{\texttt{1601.04824} \texttt{[math.LO]}}.

\bibitem{ulrich:bv}
Douglas Ulrich, \emph{Keisler's order and full Boolean-valued models}, arXiv:\href{http://arxiv.org/abs/1810.05335}{\texttt{1810.05335} \texttt{[math.LO]}}.

\end{thebibliography}
\end{document}